 \documentclass[11pt]{article}
\normalsize
\usepackage{amsfonts}
\usepackage{amsmath}
\usepackage{amssymb}
\usepackage{latexsym}

\newtheorem{thm}{Theorem}
\newtheorem{prop}[thm]{Proposition}

\newtheorem{lemma}[thm]{Lemma}

\newtheorem{remark}[thm]{Remark}

\newenvironment{proof}{\noindent  Proof:\ }{\hspace*{\fill} $\Box $\\}

\newcommand{\N}{\mathbb{N}}

\newcommand{\PSL}{\mbox{\rm PSL}}

\newcommand{\Alt}{\mbox{\rm A}}
\newcommand{\Aut}{\mbox{\rm Aut}}

\newcommand{\wt}{\mbox{\rm wt}}

\newcommand{\gen}{\mbox{\rm gen}}

\newcommand{\Fix}{\mbox{\rm Fix}}
\newcommand{\ord}{\mbox{\rm ord}}
\newcommand{\OO}{\mbox{\rm O}}

\title{\bf On the Automorphism Group of a Binary Self-dual $[120, 60, 24]$ Code}
\author{ Stefka Bouyuklieva$^1$, Javier de la Cruz$^2$ and Wolfgang Willems$^3$\\
{\tiny $^1$} University of Veliko Tarnovo, Veliko Tarnovo, Bulgaria \\
{\tiny $^2$} Universidad del Norte, Barranquilla, Colombia\\
{ \tiny $^3$} Otto-von-Guericke Universit\"at, Magdeburg, Germany\\
}
\date{}

\begin{document}
\maketitle
\begin{abstract}
We prove that an automorphism of order 3 of a putative binary self-dual $[120, 60, 24]$ code $C$ has no fixed points.
Moreover, the order of the automorphism group of $C$ divides $2^a\cdot3 \cdot5\cdot7\cdot19\cdot23\cdot29$ with $a\in \mathbb{N}_0$.
 Automorphisms  of odd composite order $r$ may occur only for $r=15$, $57$ or $r=115$
with corresponding cycle structures $3 \cdot 5$-$(0,0,8;0)$, $3\cdot 19$-$(2,0,2;0)$ or $5 \cdot 23$-$(1,0,1;0)$ respectively.
 In case that all involutions  act
 fixed point freely we have $|\Aut(C)| \leq 920$, and $\Aut(C)$ is solvable if it contains an element of prime order $p \geq 7$.
 Moreover, the alternating group $\Alt_5$ is the only non-abelian composition factor which may occur.
\end{abstract}
\section{Introduction}
Let $C=C^\perp$ be a binary self-dual code of length $n$ and minimum distance
$d$. By results of Mallows-Sloane \cite{MS} and Rains \cite{Rains},
we have
\begin{equation} \label{cota-dual}  d \leq \left\{ \begin{array}{rl} 4 \lfloor \frac{n}{24} \rfloor + 4, & \mbox{if} \ n \not\equiv 22 \bmod \, 24 \\
       4 \lfloor \frac{n}{24} \rfloor + 6, & \mbox{if} \ n \equiv 22 \bmod \ 24,
   \end{array} \right.
\end{equation}
and $C$ is called extremal if equality holds.
 Due to interesting connections with designs, extremal codes of length $24m$ are of particular interest. Unfortunately, only for $m=1$ and $m=2$ such codes are known, namely the $[24, 12, 8]$ extended Golay  code and the $[48, 24, 12]$  extended quadratic residue code  (see \cite{Pless},\cite{HLTP}). To date the existence of no other extremal code of length $24m$ is known. In numerous papers the automorphism group of a $[72, 36, 16]$, respectively a $[96, 48, 20]$ code has been studied.   In case $n=72$ only $10$ nontrivial automorphism groups may occur. The largest has order $24$
(see Theorem 1 of \cite{Borello}).  For $n=96$, only the primes $2,3$ and $5$ may divide $|\Aut(C)|$ and the cycle structure of prime  order automorphisms are
$2$-$(48,0), 3$-$(30,6), 3$-$(32,0), 5$-$(18,0)$ (see  Theorem part a) in \cite{JW}). We would like to mention here that in part b) of the Theorem (the case where
elements of order $3$ act fixed point freely)  four group orders are missing, namely  $15, 30, 240$ and $480$. The gap is due to the  fact
that the existence of elements of order $15$ with six cycles of length $15$ and two cycles of length $2$  are not excluded in the given proof.

 In his thesis the second author considered the case \cite{wir-120}. It turned out that the only primes which may divide the order of the automorphism group are $2$, $3$, $5$, $7$, $19$, $23$ and $29$.
 More precisely, if $\sigma$ is an automorphism of $C$ of prime order $p$ then  its cycle structure is given by
\begin{equation}\label{table} {
\begin{tabular}{c|c|c}
p & \mbox{\rm number of $p$-cycles} & \mbox{ \rm number of fixed points} \\
\hline
$2$ & $48,60$ & $24,0$ \\
$3$ & $ 32,34,36,38,40$ & $24,18,12,6,0 $\\
$5$ & $24$ & $0$ \\
$7$ & $17$ & $1$ \\
$19$ & $6$ & $6$ \\
$23$ & $5$ & $5$ \\
$29$ & $4$ & $4$  \\
\end{tabular} }
\end{equation}

\mbox{}\\

This paper continues the investigation of automorphisms of extremal codes of length $120$. As a main result we prove the following.

\bigskip

\noindent
{\bf Theorem} \label{main} Let $C$ be an extremal self-dual code of length $120$.
\begin{itemize}
\item[\rm a)] If $\sigma$ is an automorphism of $C$ of prime order $3$, then $\sigma$ has no fixed points.
\item[\rm b)] If $p\neq2$, then $p^2 \nmid |\Aut(C)|$. Therefore $|\Aut(C)|$ divides $ 2^a\cdot3 \cdot5\cdot7\cdot19\cdot23\cdot29$ where $a \in \mathbb{N}_0$.
\item[\rm c)] If $\sigma$ is an automorphism of $C$ of odd composite order $r$, then $r=15,57$ or $r= 115$ and the cycle structure of $\sigma$ is given by
    $15$-$(0,0,8;0)$, $57$-$(2,0,2;0)$ and $115$-$(1,0,1;0)$.
\end{itemize}

In the last section we sharpen the bound on $|\Aut(C)|$ given in part b) in case that all involutions act fixed point freely.
The largest group which may occur in this case has order $920$. Moreover, the only possible nonabelian (simple) composition factor is the
alternating group $\Alt_5$.
The proof
 uses the fact (recently shown in \cite{BW}) that
 the automorphism group of an extremal self-dual code of length $120$
does not contain elements of order $ 2\cdot 19$ and $2\cdot29$.

\section{Preliminaries}

Let $C$ be a binary code and let  $\sigma$ be an automorphism of  $C$ of odd prime order $p$.
 Suppose that $\sigma$ has $c$ cycles of length $p$ and $f$ fixed points.
To be brief we say that $\sigma$ is of type $p$-$(c;f)$.
Without loss of generality we may assume that
\begin{equation} \label{forma-auto} \sigma=(1,2, \ldots, p)(p+1,p+2, \ldots , 2p) \ldots
((c-1)p+1,(c-1)p+2, \ldots ,cp). \end{equation}

 By $\Omega_{1},\ \Omega_{2},\ldots,\ \Omega_{c}$ we denote the cycle sets and by
$\Omega_{c+1},\ \Omega_{c+2},\ldots,\ \Omega_{c+f}$  the fixed points
of $\sigma$.
Furthermore let $F_{\sigma}(C)=\{v\in C\mid v\sigma=v\}$. If $\pi:F_{\sigma}(C)\rightarrow F_{2}^{c+f}$ denotes the map
defined by $\pi(v|_{\Omega_{i}})=v_{j}$ for some $j\in \Omega_{i}$ and
$i=1,2,\ldots,c+f$, then
 $\pi(F_{\sigma}(C))$ is a binary $[c+f,\frac{c+f}{2}]$ self-dual code.
Let $C_{\pi_1}$ be the subcode of $\pi(F_\sigma(C))$ which consists of
all codewords
 which have support in the first $c$ coordinates, and let $C_{\pi_2}$ be the subcode of all codewords in $\pi(F_\sigma(C))$
   which have support in the last $f$ coordinates. Thus a generator matrix of $\pi(F_\sigma(C))$ may be written in  the form

\begin{equation}\label{form-matrix}{\rm{gen}}(\pi(F_{\sigma}(C)))=\left(%
\begin{array}{cc}
  A & O \\
  O & B \\
  D & E \\
\end{array}%
\right),
\end{equation}
where $ (A \,O)$ is a generator matrix of $C_{\pi_1}$ and $(O\, B)$ is a
generator matrix
 of $C_{\pi_2}$, $O$ being the appropriate size zero matrix.
With this notation we have

\begin{lemma} {\rm \cite{pless-libro}} \label{pless} If $k_{1}=dim \, C_{\pi_1}$ and $k_{2}=dim \, C_{\pi_2}$, then the following holds true.
\begin{enumerate}
\item[\rm a)] {\rm(}Balance Principle\rm{)} $k_{1}-\frac{c}{2}=k_{2}-\frac{f}{2}$.
\item[\rm b)] $rank(D)=rank(E)=\frac{c+f}{2}-k_{1}-k_{2}.$
\item[\rm c)] Let $\mathcal{A}$ be the code of length $c$ generated by $A$, $\mathcal{A}_{D}$ the code of length $c$ generated by
  $A$ and $D$, $\mathcal{B}$ the code of length $f$ generated by $B$, and $\mathcal{B}_{E}$ the code of length
  $f$ generated by  $B$ and $E$. Then $\mathcal{A}^\perp=\mathcal{A}_{D}$ and $\mathcal{B}^\perp=\mathcal{B}_{E}$.
\end{enumerate}
\end{lemma}

The following Lemma whose proof is trivial plays a central role when dealing with the code $\mathcal{A}$.

\begin{lemma}\label{dual-dis-1} If ${\cal A}$ is a binary linear code $[n, k]$ code with dual distance $1$, then (after a suitable permutation of the coordinates) ${\cal A}=(0|{\cal A}_1),$ where ${\cal A}_1$ is a linear $[n-1, k]$ code.  Furthermore ${\cal A}^\perp=(0|{\cal A}_1^\perp)\cup (1| {\cal A}_1^\perp)$.
\end{lemma}

\section{Cyclic structure of automorphisms of order 3}

Throughout this section let $C$ be a binary self-dual $[120,60,24]$ code.
As stated in (\ref{table})   an
automorphism of $C$ of
order 3 with $c$ cycles and $f$ fixed points satisfies $(c,f)=(32,24)$,
$(34,18)$, $(36,12)$, $(38,6)$ or $(40,0)$. We prove that only
the last case can occur; i.e., an element of order $3$ must act fixed point freely.

\begin{lemma} \label{3-32-cycles} $C$ does not have an automorphism of type $3$-$(32;24)$.
\end{lemma}
\begin{proof} Let $\sigma \in \Aut(C)$ be of type $3$-$(32;24)$.
 For $\pi(F_\sigma(C))$ we take a generator matrix in the form
(\ref{form-matrix}). By the Balance Principle (see Lemma
\ref{pless}), we get $k_{1}=k_2+4$. Since $f=d=24$ we have $k_2=0$
or $1$. \\[-2ex]

First we consider the case $k_2=0$.
In this case we have $k_1=4$ and $\pi(F_\sigma(C))$ has
a generator matrix of the form
$$\left(%
\begin{array}{cc}
  A & 0 \\
  D & E \\
\end{array}%
\right).
$$
Furthermore, $\mathcal{A}$ is a $[32, 4,d'\geq 8]$ doubly-even
code and its dual $\mathcal{A}^\perp$ has parameters $[32,28,d'^{\perp}]$.
 Looking at the online table \cite{table-mark} we see that
$d(\mathcal{A}^\perp)=d'^{\perp}\le 2$.

If $d(\mathcal{A}^\perp)=1$ we may assume (without loss of
generality) that $a_1=(100\dots 0)\in \mathcal{A}^\perp$. Thus
$\pi(F_\sigma(C))$ contains a vector $(a_1|b_1)$ with $b_1\in
\mathbb{F}_2^{24}$. Since $$\wt(\pi^{-1}(a_1|b_1))=3+\wt(b_1)\geq
24$$ we get $\wt(b_1)=21$. According to Lemma \ref{dual-dis-1},
$\mathcal{A}=(0|\mathcal{A}_1)$ and
$\mathcal{A}^\perp=(0|\mathcal{A}_1^{\perp})\cup
(1|\mathcal{A}_1^{\perp})$. The code $\mathcal{A}_1^{\perp}$ has
parameters $[31, 27]$ and by
\cite{table-mark}, its minimum distance is 1 or 2. If
$d(\mathcal{A}_1^\perp)=1$, then (up to equivalence) there is a
codeword $(010\ldots 0|b_2)\in \pi(F_\sigma(C))$ with
$\wt(b_2)=21$. But then $\wt(\pi^{-1}((a_1|b_1)+(010\ldots
0|b_2)))\leq 6+6 <24$ which contradicts the minimum distance of
$C$. If $d(\mathcal{A}_1^\perp)=2$, then (up to equivalence) there
is a codeword $(0110\ldots 0| b_2) \in \pi(F_\sigma(C))$ with
$\wt(b_2)=18$ or $22$. Since the vectors $b_1$ and $b_2$ are
orthogonal to each other, the weight of their sum $b_1+b_2$ is 1,
3, 5 or 7. But then  we obtain
$$\wt(\pi^{-1}((a_1|b_1)+(0110\ldots 0|b_2)))=
9+\wt(b_1+b_2)\leq 16 < 24,$$ a contradiction.

Next we consider the case $d(\mathcal{A}^\perp)=2$. Let
$$W_{\mathcal{A}}(y)=1+A_8y^8+A_{12}y^{12}+A_{16}y^{16}
+A_{20}y^{20}+A_{24}y^{24}+A_{28}y^{28}+A_{32}y^{32}$$ denote the weight enumerator of ${\cal A}$ and let
$$W_{\mathcal{A}^\perp}(y)=1+B_{2}y^2+B_3y^3+\ldots$$ be the weight enumerator of its dual code.
Since $k_2=0$, the code $\mathcal{A}$ does not contain the all one vector.
Hence $A_{32}=0$.

Using the power moments $$\sum_{j=d}^n A_j=2^k-1, \ \ \
\sum_{j=d}^n jA_j=2^{k-1}n, \ \ \ \sum_{j=d}^n j^2
A_j=2^{k-2}n(n+1)+2^{k-1}B_2$$ for a linear binary $[n,k,d]$ code
with $B_1=0$ (see for example \cite{pless-libro}, section 7.3) we obtain
$$A_{20}= 31 - 10A_8 - 6A_{12}- 3A_{16}+ \frac{1}{4}B_2,$$
$$A_{24}= -21 + 15A_8 + 8 A_{12} + 3 A_{16} - \frac{1}{4}B_2,$$
$$A_{28}= 5 - 6A_8 - 3 A_{12} - A_{16} + \frac{1}{4}B_2.$$
Therefore,
$ A_{24}+3A_{28}=\frac{1}{2}B_2-3A_8-A_{12}-6$ and
 $B_2$ is a multiple of $4$.
 Since $A_{j}$  are nonnegative integers, we get $B_2\ge 12$. Now we consider $a_1$,
$a_2 \in\mathcal{A}^\perp$ with $a_1\neq a_2$ and
$\wt(a_1)=\wt(a_2)=2$. Thus there are vectors $(a_i|b_i)
\in\pi(F_\sigma(C))$ with $\wt(b_i)=18$ or $22$ for $i=1,2$. In
particular,
$\wt(b_1+b_2)\leq 12$ since $b_1,b_2 \in \mathbb{F}_2^{24}$. It follows that
$$\wt(\pi^{-1}(a_1+a_2|b_1+b_2))\leq 12+ \wt(b_1+b_2)\leq 24.$$
Since the minimum distance of $C$ is 24, we get
$\wt(\pi^{-1}(a_1+a_2|b_1+b_2))=24$. Moreover $\wt(a_1+a_2)=4$,
$\wt(b_1+b_2)=12$ and $\wt(b_1)=\wt(b_2)=18$. Using this, we easily
see that $B_2\leq 4$, which contradicts the  above
inequality $B_2\geq 12$.\\[-2ex]

Finally we deal with the
case $k_2=1$.
 Now $k_1=5$ and $\mathcal{A}$ is a
doubly-even $[32, 5, d']$ code with $d'\geq 8$. By \cite{table-mark}, the dual distance satisfies
$d(\mathcal{A}^\perp)\leq2$. Thus there exist a vector $(a|b)
\in\pi(F_\sigma(C))$ with $\wt(a)\leq2$ and $\wt(b)\geq 18$. Since
$k_2=1$ we have $v=(0, \ldots, 0| \mathbf{1})\in
\pi(F_\sigma(C))$. But then $\wt(\pi^{-1}(a|b+\mathbf{1}))\leq 6+
6 < 24$, the final contradiction.
\end{proof}

\begin{lemma} \label{3-34-cycles} $C$ does not have an automorphism of type $3$-$(34;18)$.
\end{lemma}
\begin{proof} Let $\sigma$ be an automorphism of $C$ of type $3$-$(34;18)$.
Then $\pi(F_\sigma(C))$ is a self-dual $[52,26,\ge 8]$ code and we
consider a generator matrix for $\pi(F_\sigma(C))$ of the form
(\ref{form-matrix}). Since $f=18<24$ we have $k_2=0$,
hence
$$\gen(\pi(F_\sigma(C))) =\left(\begin{array}{cc}
A & O\\
D & E
\end{array}\right). 
$$
The
balance principle (see Lemma \ref{pless}) yields $k_{1}=8$.

If $(a|b)$ is a nonzero codeword in $\pi(F_\sigma(C))$, where $a$
and $b$ are vectors of length 34 and 18,  then
$3\wt(a)+\wt(b)\ge 24$ and therefore $\wt(a)\ge 2$. Clearly, $\mathcal{A}$ is a doubly-even $[34, 8, d']$ code
with $d'\geq 8$ and dual distance $d'^\perp \geq 2$.

We consider first the
case $d'^\perp=2$. If $\wt(a)=2$ then $b$ is the
all one vector of length $18$. Suppose that
$(a'|b')\in\pi(F_\sigma(C))$ is a codeword where $\wt(a')=x$
and $\wt(b')=y$ are odd numbers. Since $3x+y\equiv 0\pmod 4$ we get
$y\equiv x\pmod 4$. Thus the weight of the codeword
$\pi^{-1}(a+a'|b+b')\in C$ is $$3x+6+18-y=3x-y+24\equiv 3x-y\equiv
2x\equiv 2\pmod 4$$ or $$3x-6+18-y=3x-y+12\equiv 3x-y\equiv 2x\equiv
2\pmod 4.$$  Both cases are not possible for a doubly-even code. This shows
that in case  $d'^\perp=2$ the code $\mathcal{A}^\perp$ contains only even
weight vectors .  Hence $\mathbf{1}\in\mathcal{A}$, a
contradiction, since $C$ is doubly-even.

Thus we may assume that $d'^\perp \geq 3$.
In order to get a final contradiction we
calculate the split weight distribution
$$A_{(x,y)} =  |\{(u,w)\in \pi(F_{\sigma}(C)) \mid \wt(u)=x \; \textrm{and} \;\wt(w)=y \}|   \qquad   (0\le x\le 34, \,0\le y\le 18)   $$
of $\pi(F_\sigma(C))$.
To do so we use
the generalized MacWilliams identities
$$A_{(r, i)}=\frac{1}{2^{26}}\sum_{v=0}^{18} \sum_{w=0}^{34}A_{(w,v)}\mathcal{K}_{r}(w,34)\mathcal{K}_{i}(v,18), \ 0 \leq i\leq 18, \ 0\leq r \leq 34$$
(see \cite{Sim} and
\cite[Theorem 13]{ecuacion-split})
with the following restrictions:
 \begin{itemize}
 \item $A_{(x,y)}=0$ if $x+y$ is odd,
 \item $A_{(x,y)}=0$ if $3x+y\not\equiv 0\mod 4$,
 \item $A_{(x,y)}=0$ if $0<x+y<8$ or $0<3x+y<24$,
 \item $A_{(1,y)}=0$ and $A_{(2,y)}=0$ for $y=0,1,\dots,18$,
 \item $A_{(0,0)}=1$, $A_{(x,y)}=A_{(34-x,18-y)}$.
 \end{itemize}
By multiple substitution we find
$$A_{(9,1)}=34-22A_{(8,0)}-4A_{(12,0)},$$
$$A_{(31,3)}=20A_{(8,0)}+8A_{(12,0)}+2A_{(16,0)}-476, $$
$$A_{(20,0)}=663-10A_{(8,0)}-6A_{(12,0)}-3A_{(16,0)}.$$ Thus we obtain
$3A_{(31,3)}+2A_{(20,0)}+3A_{(9,1)}=-26A_{(8,0)}$ which forces
$A_{(8,0)}=0$ since $A_{(x,y)} \geq 0$. Thus $0=A_{(9,1)}=34-4A_{(12,0)}$ which is not possible.
\end{proof}

 \begin{lemma} \label{3-36-cycles} $C$ does not have an automorphism of type $3$-$(36;12)$.
\end{lemma}
\begin{proof} Let $\sigma$ be an automorphism of $C$ of type $3$-$(36;12)$.
Thus $\pi(F_{\sigma}(C))$ is a self-dual $[48,24,\ge 8]$ code.

We take again a generator matrix for $\pi(F_\sigma(C))$ in the form
(\ref{form-matrix}). Since $f<24$, we have $k_2=0$  and by the Balance
Principle (see Lemma \ref{pless}), we get $k_{1}=12$. Hence
$\pi(F_\sigma(C))$ has a generator matrix of the form
$$\left(%
\begin{array}{cc}
  A & O \\
  D & E \\
\end{array}%
\right).
$$
 Note that $\mathcal{A}$ is a doubly-even $[36, 12, d']$ code with $d'\geq 8$. If $a\in \mathcal{A}^\perp,$ then there exists a vector
 $(a|b)\in \pi(F_\sigma(C))$ with $3\wt(a)+\wt(b)\geq 24$ and $\wt(b)\leq 12$. Thus $\wt(a)\geq 4$ and the dual distance $d'$ of ${\cal A}$ satisfies
  $d'^\perp \geq 4$.
 A calculation of the coefficients $A_{(x,y)}$  $(x=0,1,\dots,36,\, y=0,1,\dots,12) $ of the split weight enumerator of
 $\pi(F_\sigma(C))$  yields $$A_{(28,0)}=7092+39A_{(8,0)}-4A_{(16,0)}
 \ \  \mbox{and} \ \ A_{(32,0)}=A_{(16,0)}-10A_{(8,0)}-1773.$$
Thus $A_{(28,0)}+4A_{(32,0)}=-A_{(8,0)}$. This implies $A_{(8,0)}=0$, hence $A_{(28,0)}=A_{(32,0)}=0$ and
$A_{(16,0)}=1773$. But then
$$A_{(30,2)}=18A_{(16,0)}- 192A_{(8,0)}-32076=18A_{(16,0)}-32076=-162<0,$$
a contradiction.
\end{proof}

\begin{lemma} \label{3-38-cycles} $C$ does not have an automorphism of type $3$-$(38;6)$.
\end{lemma}
\begin{proof} Let $\sigma \in \Aut(C)$ be of type $3$-$(38;6)$. Now  $\pi(F_\sigma(C))$ is a
self-dual $[44,22,d_\pi]$ code. According to (\ref{cota-dual}) we
have $d_\pi \leq 8$. If
 $d_\pi = x +y,$ where $x$ is the number of 1's in the first $c$ coordinates  and $y$ is the number of 1's in the last $f$ coordinates
 of a minimal weight codeword in $\pi(F_\sigma(C))$, then $x+y \leq 8$ and $3x + y \geq 24$.
 This forces $ x \geq 8$, $y=0$ and $d_\pi=8$. Thus $\pi(F_\sigma(C))$ is a
self-dual $[44,22,8]$ code. According to \cite{conway-sloane}  there are two possible weight
enumerators for such a code, namely
$$W_1(y)=1+(44+4\beta)y^8+(976-8\beta)y^{10}+\ldots $$ where $10\leq\beta \leq122$ \; and
$$W_2(y)=1+(44+4\beta)y^8+(1232-8\beta)y^{10}+(10241-20\beta)y^{12}\ldots $$ where $0\leq\beta \leq154$.

Now we take a generator matrix for $\pi(F_\sigma(C))$ in the form of
(\ref{form-matrix}). Since $f<24,$ we have  $k_2=0$  and by the Balance Principle (see Lemma \ref{pless}), we get
$k_{1}=16$. Hence a generator matrix of $\pi(F_\sigma(C))$ is
 of the form
$$\left(%
\begin{array}{cc}
  A & O \\
  D & E \\
\end{array}%
\right).
$$
Observe that $\mathcal{A}$ is a $[38, 16, d']$ doubly-even code with
$d'\geq 8$. Since $d_\pi=8$ there is a vector $(u| w) \in
\pi(F_\sigma(C))$
with $\wt(u|w)=8$ and $3\wt(u)+\wt(w)\geq24$.
This implies $\wt(u)=8$ and $\wt(w)=0$, hence $d'=8$.

On the other hand, if $a\in \mathcal{A}^\perp$, then there exists a vector $(a|b)\in \pi(F_\sigma(C))$ with $3\wt(a)+\wt(b)\geq 24$ and $\wt(b)\leq 6$.
Hence $\wt(a)\geq 6$. Consequently $\mathcal{A}$ is a $[38, 16, 8]$ doubly-even code with dual distance $d'^\perp \geq 6$.
Furthermore, ${\cal A}$ does not contain a codeword of weight $36$ since for $(u|0)\in \pi(F_\sigma(C))$ with $\wt(u)=36$ we get $$\wt(\pi^{-1}(u+\mathbf{1}|\mathbf{1}))\leq 6+6< 24.$$

Now let $$W_{\mathcal{A}}(y)=1+A_8y^8+A_{12}y^{12}+\ldots+A_{32}y^{32}$$ and $$W_{\mathcal{A^\perp}}(y)=1+A_6^\perp y^6+A_7^\perp y^7+\ldots$$
denote the weight enumerators of ${\cal A}$ and ${\cal A^\perp}$.
Using the MacWilliams identity equations and Maple calculations we get
$$A_{12}=2808-6A_8, \ \ldots, \ A_{28}=632-6A_{8},\;
A_{32}=-27+A_{8}$$ and $$A_{6}^\perp=4A_{8}-87,\;
A_{7}^\perp=480-8A_{8},\; A_{8}^\perp=660+4A_{8},\;
A_{9}^\perp=1920,\; A_{10}^\perp=7952-24A_{8},\, \dots$$
To finish the proof we also need the weight enumerator $W_{\pi(F_\sigma(C))}(y)=\sum A_i^\pi y^i$ of $\pi(F_\sigma(C))$.
Note that $A_{8}=A_8^\pi$.

Since
$A_7^\perp=480-8A_{8} \geq 0$, we obtain $A_8=A_8^\pi=44+4\beta \leq60$.
Hence $0\leq \beta \leq 4$ which shows that $W_2$ is the weight enumerator of
$\pi(F_\sigma(C))$.

On the other hand,
$$ A_{12}^\pi=A_{(12,0)}+A_{(10,2)}+A_{(8,4)}+A_{(6,6)},$$
where\\

\noindent
$A_{(12,0)}=A_{12}=2808-6A_8=2544-26\beta,$\\
$A_{(10,2)}=(A_{(10,2)}+A_{(10,6)})-A_{(10,6)}=A_{10}^\perp-A_{(28,0)}=A_{10}^\perp-A_{28}=7320-18A_8=6528-72\beta,$\\
$A_{(8,4)}=(A_{(8,4)}+A_{(8,0)})-A_{(8,0)}=A_{8}^\perp-A_{8}=660+3A_8=792+12\beta$ and\\ $A_{(6,6)}=A_{(32,0)}=A_{32}=-27+A_8=17+4\beta$.\\

\noindent
It follows $A_{12}^\pi=9881-82\beta$. Computing this coefficient again via $W_2(y)$ we get $A_{12}^\pi=10241-20\beta$, a contradiction.
\end{proof}

So far we have shown that automorphisms of order $3$ act fixed point freely on the coordinates of
$C$ which completes part a) of the Theorem.

\section{Order of the automorphism group and automorphisms of composite order}

 In this section we prove part b) of the Theorem.

\begin{prop} 
 \label{cuadrado-divisor}
Let $C$ be a binary code of length $n$. Suppose that for every
automorphism of $C$ of prime order $p$ the number of $p$-cycles is not divisible by $p$ and
the number $f$ of fixed points satisfies $f < p$. Then
$ p^2 \nmid |\Aut(C)|$.
\end{prop}
\begin{proof}
Suppose that $p^2 \mid |\Aut(C)|$. Thus, by Sylow's Theorem,
there exists a subgroup $ N \leq \Aut(C)$ with $|N|=p^2$, which must be abelian.
If there is an automorphism, say $\sigma$, of order $p^2$, then the number of $p$-cycles of $\sigma^p$ is divisible
by $p$, a contradiction.
 Thus we
may assume that all non-trivial elements in $N$ have order $p$. In
particular, $N = \langle \sigma, \theta \rangle$. Since $\sigma$ and
$\theta$ commute $\sigma$ acts on the orbits of size $p$ of
$\theta$. By assumption, the number of such orbits is
not divisible by $p$. Thus $\sigma$ fixes the elements of at least
one orbit of $\theta$, say $\Omega$. It follows that $\theta =
\sigma^k$ on $\Omega$ for some $k \in \N$. Thus $\theta
\sigma^{-k}$, which is not the identity on the $n$ coordinates, has at least $p$ fixed
points, a contradiction.

\end{proof}

Applying this in the particular situation of a binary self-dual
extremal code of length $120$ we get

\begin{prop} Let $ C $ be a binary self-dual  code with parameters $ [120, 60, 24] $. Then $|\Aut(C)|$ divides $2^a\cdot3 \cdot5\cdot7\cdot19\cdot23\cdot29$, where $a \in \mathbb{N}_0$.
\end{prop}
\begin{proof} Suppose that $p \mid |\Aut(C)|$, where $p\geq3$ is a prime. Then, according to  (\ref{table}) and part a) of the Theorem,  we have $(c,f)=(40, 0),(24, 0), (17, 1),
(6,6), (5,5), (4,4)$. Thus Proposition
\ref{cuadrado-divisor} implies $ p^2 \nmid |\Aut(C)|$.
\end{proof}

Let $\sigma$ be an automorphism of $C$ of order $p\cdot r$ where
$p,r$ are primes. We say that $\sigma$ is of type  $p\cdot
r$-$(s_{1}, s_{2}, s_{3}; f)$ if $\sigma$ has $s_1$ $p$-cycles,
$s_2$ $r$-cycles, $s_3$ $pr$-cycles and $f$ fixed points. In
particular, $n = s_1p +s_2r +s_3pr + f$. In the special case $p=r$
we write $p^2$-$(s_1,s_2;f)$ where $n=s_1p +s_2p^2+f$.

\begin{lemma} {\rm \cite{Radinka}} \label{compuesto} Let $C$ be a self-dual code and let $p,r$ be different odd primes.
\begin{itemize}
\item[a)]
 If $C$ has an automorphism of type $p\cdot r$-$(s_{1}, s_{2}, s_{3}; f)$, then the automorphism  $\sigma^r$ is of type $p$-$(s_{1}+s_{3}r; s_{2}r+f)$ and $\sigma^p$ is of
 type $r$-$(s_{2}+s_{3}p; s_{1}p+f)$.
\item[b)] If $C$ has an automorphism  of type $p^2$-$(s_{1},s_{2}; f)$,  then $\sigma^p$ is of type $p$-$(s_{2}p; s_{1}p+f)$.
\end{itemize}
\end{lemma}

Since by  Proposition \ref{cuadrado-divisor}) there are no automorphisms of order $p^2$ for $p$ an odd prime, the following  completes the proof of the Theorem.

\begin{lemma} \label{multiplo} If $\sigma$ is an automorphism of a self-dual $[120,60,24]$ code $C$ of order $p\cdot r$ where $p$ and $r$ are different odd primes, then the order of $\sigma$ is $3 \cdot 5$, $3 \cdot 19$ or $5 \cdot 23$  and its cycle structure is given by $3 \cdot 5$-$(0,0,8;0)$, $3\cdot 19$-$(2,0,2;0)$ or $5 \cdot 23$-$(1,0,1;0)$.
\end{lemma}
\begin{proof}
Let $3\le p<r\le 29$. In order to prove the Lemma we distinguish three cases. \\
Case $p=3$:\\
In this case $\sigma^r$ is an automorphism of type $3$-$(s_{1}+s_{3}r;
s_{2}r+f)$.  Thus $s_2=f=0$ and $s_1+s_3r=40$, since we proved already that elements of order $3$ have no fixed points. Thus
$\sigma^3$ is of type $r$-$(3s_{3}; 3s_{1})$. According
to  (\ref{table}), we get $r=5$, $s_3=8$, $s_1=0$, or $r=19$,
$s_3=s_1=2$. It follows that  $\sigma$ is of type $3 \cdot
5$-$(0,0,8;0)$ or $3\cdot 19$-$(2,0,2;0)$. \\
Case $p=5$: \\
Now $\sigma^r$ is an automorphism of type $5$-$(s_{1}+s_{3}r;
s_{2}r+f)$ and therefore $s_2=f=0$, $s_1+s_3r=24$, since
elements of order $5$ also have no fixed points.  Thus
$\sigma^5$ is of type $r$-$(5s_{3}; 5s_{1})$. Looking again at
 (\ref{table}), we see that  $r=23$ and $s_3=s_1=1$ is the only possibility.
It follows that $\sigma$ is of type $5 \cdot 23$-$(1,0,1;0)$. \\
Final case  $p>5$: \\
Now $\sigma^r$ is an automorphism of type $p$-$(s_{1}+s_{3}r;
s_{2}r+f)$ and the data in  (\ref{table}) lead to $s_2r+f=1$, 4, 5 or 6.
Since $r \geq 19$ we obtain $s_2=0$. Thus
$\sigma^p$ is of
 type $r$-$(s_{3}p; s_{1}p+f)$ where $s_3p=4$, 5 or 6, which is
 not possible as $p>5$. This proves that there are no possible
 automorphisms in this case.
\end{proof}

\section{The structure of the automorphism group if all involutions act fixed point freely}

The first author proved in \cite{Stefka} that
 involutions of the automorphism group of
 a binary self-dual extremal code $C$ of length $n=24m > 24$ permute the $n$ coordinates without fixed points unless $n = 120$, the case we are considering in
 this paper.
 In the exceptional case  involutions have no fixed points or exactly $24$. Throughout this section we assume
that all involutions act fixed point freely. In this case the Theorem and the list in (\ref{table}) show that all automorphisms have a unique
cycle structure. This enables us to compute the order of $G=\Aut(C)$ via the Cauchy-Frobenius lemma (\cite{Isaacs}, 1A.6) which says that
$$ t=\frac{1}{|G|}\sum_{g\in G}|\Fix(g)|$$
is the number of orbits of $G$ on the coordinates of $C$. Here $\Fix(g)$ denotes the number of fixed points of $g$. In order to compute $t$ we only need to determine the number of
automorphisms of prime order $p$ for $p \geq 7$ since only those have fixed points assuming that
involutions are fixed point free.

Let $\tau_p \in G$ of prime order $p \geq 7$. According to Sylow's theorem the number of Sylow $p$-subgroups is given by
$$  n_p = |G:N_G(\langle \tau_p \rangle)| \equiv 1 \  (\bmod  p).$$

If $\sigma\in N_{G}(\langle \tau_{p} \rangle)$ is an automorphism of prime order
$r \not= p$ then
$\sigma\tau_p\sigma^{-1}=\tau_p^s$
for some
 integer $0   \leq s<p$. Hence $\sigma$ acts on the set  $T=\{\Omega_{c+1},\dots, \Omega_{c+f}\}$  of fixed
points of $\tau_p$. Since $\ord(\sigma|_T)\mid \ord(\sigma)=r$ and
$\ord(\sigma|_T)\le f\le 6$ (according to the Theorem and the list
in (\ref{table})), we see that $r=2$, 3, 5 or $\ord(\sigma|_T)=1$.
Finally the $2$-part $|G|_2$ of $|G|$ is bounded by $8$ since a
Sylow $2$-subgroup of $G$ acts regularly on the coordinates in the
considered case.

\begin{lemma} \label{normalizer} \mbox{} We have
 \begin{itemize}
  \item[\rm a)] $n_{29} =1,\, 2\cdot3\cdot5,\, 2^2\cdot3\cdot7\cdot 19,\, 2^3\cdot3\cdot 23,\, 2^2\cdot 5 \cdot 7 \cdot 23$ or $3\cdot5\cdot19\cdot23$.
  \item[\rm b)] $n_{23} = 1, \, 2^3\cdot3, \, 2\cdot5\cdot7, \ 2^2\cdot29, \, 2^3\cdot 5\cdot19$ or  $2^3\cdot5\cdot7\cdot29$.
  \item[\rm c)] $n_{19}=1, \, 5\cdot 23, \,  3\cdot 7\cdot 29, \, 3\cdot 5\cdot 7\cdot
23\cdot 29, \,  2\cdot 29, \, 2\cdot 5\cdot 23\cdot 29, \, 2\cdot
3\cdot 5\cdot 7, \, 2^2\cdot 5, \, 2^2\cdot 3\cdot 5\cdot 7\cdot
29, \,  2^3\cdot 5\cdot 29$ \\ \mbox{} \ \qquad or $2^3\cdot 3\cdot 23$.
 \end{itemize}
\end{lemma}
\begin{proof} a) First observe that  $\tau_p$ has exactly $f=4$ fixed points. Therefore $r=2$.
Hence $n_{29}=\frac{|G|}{2^x\cdot 29}=2^{a-x}\cdot 3^b\cdot
5^c\cdot 7^d\cdot 19^e\cdot 23^f$. Since $n_{29}\equiv 1\pmod
{29}$ we obtain exactly the six possibilities mentioned in a). \\
b) In this case we have $f=5$ and therefore $r=5$.
Hence $$n_{23}=\frac{|G|}{5^y\cdot 23}=2^{a}\cdot 3^b\cdot
5^{c-y}\cdot 7^d\cdot 19^e\cdot 29^g.$$ Since $n_{23}\equiv 1\pmod
{23}$ exactly the six possibilities mentioned in b) may occur. \\
c) Now $f=6$ and therefore $r=2$ or $r=3$.
Hence $$n_{19}=\frac{|G|}{2^x\cdot 3^z\cdot 19}=2^{a-x}\cdot
3^{b-z}\cdot 5^{c}\cdot 7^d\cdot 23^f\cdot 29^g.$$
The congruence
$n_{19}\equiv 1\pmod {19}$ leads to the 11 possibilities in c).
\end{proof}

\begin{lemma} \label{order} \mbox{}
  \begin{itemize}
    \item[\rm a)] If $29 \mid |G|$ then $|G| = 2^a\cdot29$ or $|G|=2^a\cdot3\cdot5\cdot29$ where $0 \leq a \leq 3$.
    \item[\rm b)] If $23 \mid |G|$ then $|G| = 5^c\cdot 23$ or $|G| = 2^3 \cdot3\cdot 5^c\cdot 23$ where $c=0,1$.
    \item[\rm c)] If $19 \mid |G|$ then $|G|= 2^a\cdot3^b \cdot 19$ or $|G| = 2^a\cdot 3^b \cdot 5 \cdot 19$ where $0 \leq a \leq 3$ and $b=0,1$.
    \item[\rm d)] If $7 \mid |G|$ then $|G|= 7$ or $2^3\cdot7$.
  \end{itemize}
\end{lemma}
\begin{proof} a) Using Lemma \ref{normalizer}, we see that $|G|= 2^a\cdot 29$, $2^a\cdot 3\cdot 5\cdot 29$,
$2^a\cdot 3\cdot 7\cdot 19\cdot 29$, $2^3\cdot 3\cdot 23\cdot 29$,
$2^a\cdot 5\cdot 7\cdot 23\cdot 29$ or $2^a\cdot 3\cdot 5\cdot
19\cdot 23\cdot 29$.
In the last three cases we have
$n_{23}=2^3\cdot 3\cdot 29$, $2^a\cdot 5^{1-y}\cdot 7\cdot 29$, or
$2^a\cdot 3\cdot 5^{1-y}\cdot 19\cdot 29$. Since $n_{23} \equiv 1 \, (\bmod 23)$ only $n_{23}=2^3\cdot
5\cdot 7\cdot 29$ is possible which leads to $|G|=2^3\cdot 5\cdot 7\cdot
23\cdot 29$.
But in this case $n_7=2^3\cdot 5\cdot 23\cdot 29\equiv
3\pmod 7$, a contradiction. Thus
 23 does not divide $|G|$.
 If $|G|=2^a\cdot 3\cdot 7\cdot
19\cdot 29$ then $n_{19}=2^{a-x}\cdot 3^{1-y}\cdot 7\cdot 29$.
Looking at the possibilities in Lemma \ref{normalizer} we see that $n_{19}=3\cdot
7\cdot 29$. For $n_7$ we get $n_7=2^a\cdot 3\cdot 19\cdot
29\equiv 2^a\pmod 7 \equiv 1\pmod 7$, hence
$a=3$ since $a \geq 2$ in this case.

Applying the Cauchy Frobenius lemma we obtain
$$ \begin{array}{rcl}
t & = &  \frac{120+6n_7+6\cdot 18n_{19}+4\cdot 28n_{29}}{2^3\cdot 3\cdot 7\cdot
19\cdot 29} \\[2ex]
 & = & \frac{120+6\cdot 2^3\cdot 3\cdot 19\cdot
29+6\cdot 18\cdot 3\cdot 7\cdot 29+4\cdot 28\cdot 2^2\cdot 3\cdot
7\cdot 19}{2^3\cdot 3\cdot 7\cdot 19\cdot 29}=\frac{7}{2},
\end{array} $$ a contradiction. Therefore only the first two cases are possible, namely
$|G|=2^a\cdot 29$ or $2^a\cdot 3\cdot
5\cdot 29$ where $a=0,1,2,3$. \\
b) First note that $ 29 \nmid |G|$ as shown above. Hence $n_{23}=1$, $2^3\cdot3$,
$2\cdot5\cdot7$ or $2^3\cdot 5\cdot 19$, by Lemma \ref{normalizer}. Thus $|G|=5^c\cdot 23$, $2^3\cdot3\cdot
5^c\cdot 23$, $2\cdot5\cdot7\cdot 23$ or $2^3\cdot 5\cdot 19\cdot 23$. In
the last case $n_{19}=2^{3-x}\cdot 5\cdot 23$ which froces
$n_{19}=5\cdot 23$. It follows
$$t=\frac{120+6\cdot 18n_{19}+5\cdot
22n_{23}}{2^3\cdot 5\cdot 19\cdot 23}=\frac{120+6\cdot 18\cdot
5\cdot 23+5\cdot 22\cdot 2^3\cdot 5\cdot 19}{2^3\cdot 5\cdot
19\cdot 23}=\frac{11}{2},$$
a contradiction.  If
$|G|=2\cdot5\cdot7\cdot 23$ then $n_7=230\equiv 6\pmod 7$, a contradiction again.
Thus $|G|=5^c\cdot 23$ or
$2^3\cdot3\cdot 5^c\cdot 23$ where $c=0,1$. \\
c) In this case both 23 and 29 do not divide $|G|$. Thus according to Lemma \ref{normalizer} we have $n_{19}=1$,
$2\cdot 3\cdot 5\cdot 7$ or $2^2\cdot 5$. It follows that $|G|=2^a\cdot
3^b\cdot 19$, $2^a\cdot 3\cdot 5\cdot 7\cdot 19$ or $2^a\cdot
3^b\cdot 5\cdot 19$. In the second case we have $n_7=2^a\cdot
3\cdot 5\cdot 19\equiv 2^a\cdot 5\not\equiv 1\pmod 7$ for $0\le
a\le 3$. Thus $|G|=2^a\cdot
3^b\cdot 19$ or $2^a\cdot 3^b\cdot 5\cdot 19$ where $a=0,1,2,3$ and
$b=0,1$.\\
d) By a), b) and c) we see that $G$ is a $\{2,3,5,7\}$-group. Since an element of order
$7$ has exactly one fix point we get $n_7 = \frac{|G|}{7}$.
If $|G|=2^a3^b5^c7$ then the Cauchy-Frobenius Lemma
yields
 $$t=\frac{1}{2^a3^b5^c7}(120+\sum_{\ord(g)=7}1)=\frac{1}{2^a3^b5^c7}(120+6n_7) = \frac{120}{2^a3^b5^c7}+\frac{6}{7}$$
and $t \in \N$ forces
$$(a,b,c) = (0,0,0), (3,0,0), (0,1,1), (3,1,1).$$
If $(a, b, c)=(0, 1,1)$ then $|G|=3\cdot5\cdot7=105$. Using MAGMA we see that there are exactly two groups of order $105$, all with $|N_{G}(\langle\tau_{7}\rangle)|=105 \neq 7$.
In the latter case $(a, b, c)=(3, 1, 1)$ we have $| G |=840$ and  Magma shows that there are exactly 186 groups of order 840, all with $|N_{G}(\langle\tau_{7}\rangle)|=105, 840 \neq 7$.
Therefore $|G|=7$ or $56$ .
\end{proof}

\begin{lemma} \label{simple} The only nonabelian composition factor which possibly  occurs in $\Aut(C)$ is the alternating group $\Alt_5$.
\end{lemma}
\begin{proof}
Let $H$ be a  nonabelian composition factor of $G$. If $G$ is a $\{2,3,5\}$-group then $|G| \mid 2^3\cdot 3 \cdot 5 =120$ and $H$ must
be isomorphic to $\Alt_5$. Thus we may assume that $ p \mid |G|$ where $p=7,19,23$ or $29$. By Lemma \ref{order}, we have $|G| \leq 3480$.
According to the classification of finite simple nonabelian groups, $H$ must be a group in the following list.
$$ \Alt_5, \, \Alt_6, \,  \PSL(2,8), \, \PSL(2,11), \, \PSL(2,13), \, \PSL(2,17), \,  \Alt_7, \, \PSL(2,19)$$
Note that $\PSL(2,11), \PSL(2,13)$ and $\PSL(2,17)$ can not occur since neither $11, 13$ nor $17$  divide $|G|$.
Furthermore $\Alt_6, \Alt_7, \PSL(2,8), \PSL(2,19)$ are not possible since $ 3^2 \nmid |G|$. Thus only the group $\Alt_5$ is left.
\end{proof}

To sharpen the results of Lemma \ref{order} we need the following fact.

\begin{lemma}  \label{B-W} {\rm \cite{BW}}  The automorphism group of an extremal self-dual code of length $120$ does not contain
 elements of order $ 2\cdot 19$ and $2\cdot29$, independent whether involutions have fixed points or not.
\end{lemma}

\begin{prop} \label{order2} Let $G=\Aut(C)$ where $C$ is an extremal self-dual code of length $120$. Suppose that all  involutions of $G$ act
fixed point freely. Then we have.

\begin{itemize}
\item[\rm a)] If $29 \mid |G|$ then $|G| = 2^a\cdot29$  where $0 \leq a \leq 2$.
    \item[\rm b)] If $23 \mid |G|$ then $|G| = 5^c\cdot 23$ or $|G| = 2^3 \cdot 5^c\cdot 23$ where $c=0,1$.
    \item[\rm c)] If $19 \mid |G|$ then $|G|= 2^a\cdot3^b \cdot 19$ where $0 \leq a,b \leq 1$
     \item[\rm d)] If $7 \mid |G|$ then $|G|= 7$ or $2^3\cdot7$.
    \item[\rm e)] If $G$ is a $\{2,3,5\}$-group then $|G| \leq 120$.
\end{itemize}
\end{prop}
\begin{proof} In the proof we use the common notation $\OO_p(G)$ for the largest normal $p$-subgroup of $G$ \\
a) By Lemma \ref{order}, we may suppose that $|G| = 2^a\cdot3\cdot5\cdot29$ where $0 \leq a \leq 3$.
If $ \OO_p(G) \not= 1$ for $p=3,5$ or $29$ then $G$ contains elements of order $3\cdot29$ or $ 5 \cdot 29$ in contrast to  the Theorem.
Thus $p=2$ and there is an element of order $2 \cdot 29$ which contradicts Lemma \ref{B-W}. The only possibility left is that $\Alt_5$ is a normal
subgroup in $G$ according to Lemma \ref{simple}. In this case we have again an element of order $2 \cdot 29$, hence a contradiction. It follows that $|G| = 2^a\cdot29$ with
$0 \leq a \leq 2$. Note that in case $a=3$ there is an element of order $2 \cdot 29$. \\
b)  This is part c) of Lemma \ref{order}. \\
c) According to Lemma \ref{order}, we first consider the case $|G| = 2^a \cdot 3^b \cdot 19$ with $0 \leq a \leq 3$ and $b=0,1$.
Suppose that $a=2$ or $a=3$. Clearly, $\OO_2(G)=1$ otherwise there is an element of order $2\cdot 19$ in contrast to
Lemma \ref{B-W}. Furthermore $\OO_{19}(G) =1$ otherwise we get the same contradiction. Thus $\OO_3(G) \not=1$ since $G$ is solvable, and we get
an element of order $3 \cdot 19$. It follows $n_{19} = 2^x \equiv 1 (\bmod 19)$ with $x=1,2$, a contradiction. Thus $|G|=2^a\cdot3^b \cdot 19$ where $0 \leq a,b \leq 1$.

Now suppose, according to Lemma \ref{order}, that $|G| = 2^a\cdot 3^b \cdot 5 \cdot 19$ where $0 \leq a \leq 3$ and $b=0,1$.
Suppose that $3 \mid |G|$.
Clearly, $\OO_p(G) = 1$ for $p=5$ and $p=19$ since otherwise there exists an element of order $5 \cdot 19$, in contrast to the
Theorem. Furthermore $\OO_2(G)=1$ since there are no elements of order $2\cdot 19$, by Lemma \ref{B-W}.
If $\OO_3(G) \not=1$ then $G$ is solvable. Thus there exists a $\{5,19\}$-Hall subgroup. But such a group is cyclic, i.e.
there is an element of order $5 \cdot 19$, a contradiction to the Theorem again. Finally, if $\Alt_5$ is involved in $\Aut(C)$
then it must be a normal subgroup of $\Aut(C)$ and elements of order $19$ centralize $\Alt_5$, a contradiction.
  This shows that $3 \nmid |G|$ in the considered case.  Thus $|G|= 2^a  \cdot 5 \cdot 19$
and $G$ is solvable. Since $\OO_2(G)=1$ we get an element of order $5 \cdot 19$, a contradiction to the Theorem.
In summary, the case $|G|= 2^a \cdot 3^b \cdot 5 \cdot 19$ does not occur.
\end{proof}

\begin{remark} {\rm
a) Lemma \ref{simple} and Proposition \ref{order2} show that $\Aut(C)$ is solvable if a prime $p \geq 7$ divides $|G|$. \\
b) The largest group occurring in Proposition \ref{order2}  has order $920$. \\
c) In case a) the Sylow $29$-subgroup must be normal, in case c)
the Sylow $2$-subgroup is elementary abelian and normal.
  }
 \end{remark}

\end{document}